\documentclass{lms}
\usepackage{graphicx} 
\usepackage{epstopdf}
\usepackage{amsmath}

\newcommand*{\suchthat}[1]{\,\left|\, #1 \right.}
\newcommand{\argmax}{\operatornamewithlimits{argmax}}

\newtheorem{theorem}{Theorem}[section] % 1st argument is your name for it
\newtheorem{lemma}[theorem]{Lemma}     % 2nd argument is what is printed
\newtheorem{corollary}[theorem]{Corollary}

\newnumbered{assertion}{Assertion}    % 1st argument is your name for it
\newnumbered{conjecture}{Conjecture}  % 2nd argument is what is printed
\newnumbered{definition}{Definition}
\newnumbered{fact}{Fact}
\newnumbered{hypothesis}{Hypothesis}
\newnumbered{remark}{Remark}
\newnumbered{note}{Note}
\newnumbered{observation}{Observation}
\newnumbered{problem}{Problem}
\newnumbered{question}{Question}
\newnumbered{algorithm}{Algorithm}
\newnumbered{example}{Example}
\newunnumbered{notation}{Notation} % This is usually unnumbered
\title[Divides Relation  Ordered by Prime Signatures]% end with percent
{Graph Invariants Based on the Divides Relation and Ordered by Prime Signatures}
\author{Sung-Hyuk Cha, Edgar G. DuCasse, and Louis V. Quintas}

\classno{06C05 (primary), 11B99 (secondary)}
\begin{document}
\maketitle

\begin{abstract}
Directed acyclic graphs whose nodes are all the divisors of a positive integer $n$ and arcs $(a,b)$ defined by  $a$ divides $b$ are considered.
Fourteen graph invariants such as order, size, and the number of paths are investigated for two classic graphs, the {\em Hasse diagram} $G^H(n)$ and its {\em transitive closure}  $G^T(n)$ derived from the divides relation partial order. 
Concise formulae and algorithms are devised for these graph invariants and several important properties of these graphs are formally proven.
Integer sequences of these invariants in natural order by $n$ are computed and several new sequences are identified by comparing them to existing sequences in the On-Line Encyclopedia of Integer Sequences. These new and existing integer sequences are interpreted from the graph theory point of view.
Both $G^H(n)$ and $G^T(n)$ are characterized by the prime signature of $n$. Hence, two conventional orders of prime signatures, namely the {\em graded colexicographic} and the {\em canonical} orders are considered and additional new integer sequences are discovered.

\end{abstract}

% use lowercase except for proper names

\section{Introduction} % use lowercase except for proper names
\label{s1}
Let $V(n)$ be the set of all positive divisors of a positive integer  $n$ as defined in~(\ref{e01}). 
For instance, $V(20) = \{1, 2, 4, 5, 10, 20\}$. 
The partial order called the  {\em divides} relation, $a$ divides $b$  denoted $a|b$, is applied to $V(n)$ and 
yields two types of directed acyclic graphs (henceforth referred  simply as graphs)  as shown in Figure~\ref{f01}.  
\begin{figure}[htb]
\begin{center}
\begin{tabular}{cc}
\resizebox{!}{1.2in}{\includegraphics{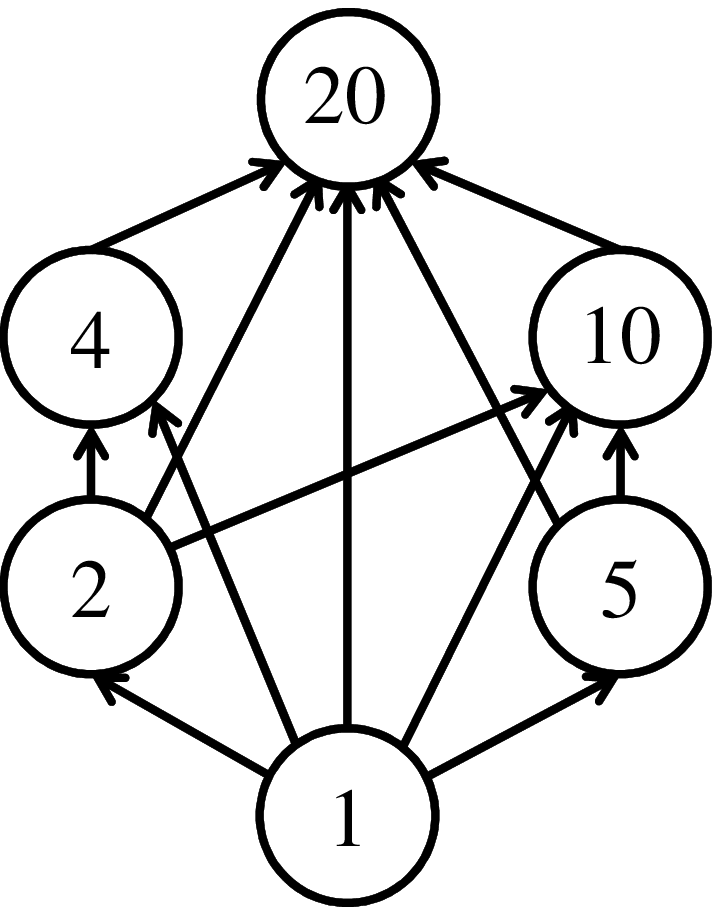}}
&
\resizebox{!}{1.2in}{\includegraphics{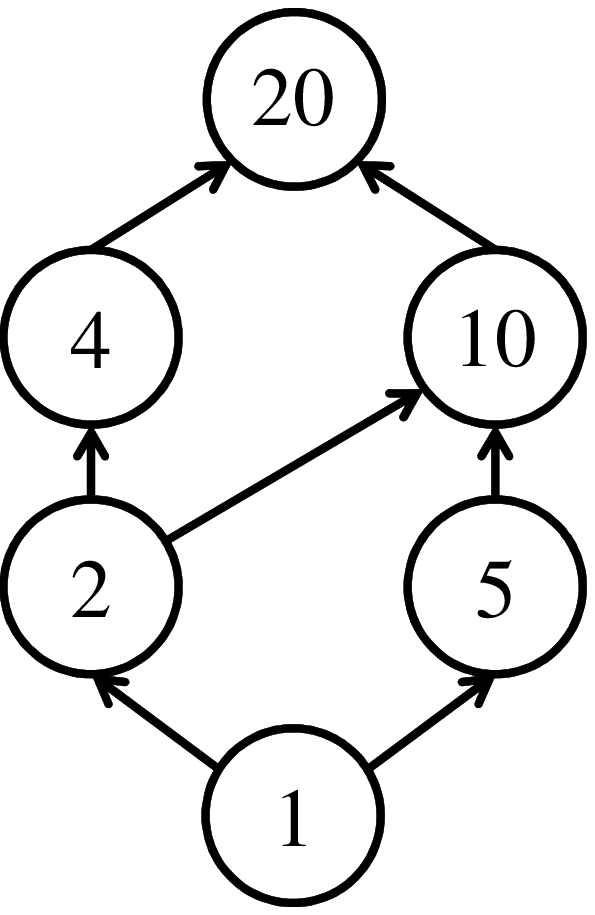}}\\
%\resizebox{!}{1.5in}{\includegraphics{f01c.eps}}\\
(a) Transitive Closure $G^T(20)$  & (b) Hasse diagram $G^H(20)$  
\end{tabular}
\end{center}
\caption{\label{f01} Two basic graphs derived from the divides relation.}
\end{figure}
The  first graph is called the {\em transitive closure}, $G^T(n) = (V(n), E^T(n))$ where 
\begin{equation}
V(n) = \{x  \suchthat{ x  \in Z^{+} \wedge x|n}\}
\label{e01}
\end{equation}
\begin{equation}
E^T(n) = \{(a,b)  \suchthat{a,b  \in V(n) \wedge a < b \wedge a|b}\}
\label{e02}
\end{equation}
Next, when all arcs in $G^T(n)$ with alternative transitive paths are excluded, the graph becomes a {\em Hasse diagram}  denoted as  $G^H(n) = (V(n), E^H(n))$  where $E^H(n)$ is defined in~(\ref{e03}).
\begin{equation}
E^H(n) = E^T(n) - \{(a,b) \in E^T(n) \suchthat{\exists c  \in V(n) (a < c <  b \wedge a|c \wedge c|b)}\}
\label{e03}
\end{equation}
Figures~\ref{f01} (a) and (b) show the {\em Transitive Closure} $G^T(20)$ and  {\em Hasse diagram} $G^H(20)$ , respectively. 
Note that $G^H(n) = G^T(n)$ if and only if $n$ is a prime. 
Numerous integer sequences have been discovered from the divides relation from the number theory point of view (see ~\cite{oeis}). 
 In Section~\ref{s2},   this paper not only compiles 
various existing integer sequences in~\cite{oeis}, but also discovers numerous integer sequences from the graph theory point of view, mainly from $G^H(n)$ and $G^T(n)$.

By the {\em Fundamental Theorem of Arithmetic}, every positive integer $n > 1$ can be represented by $\omega$ distinct prime numbers $p_1, p_2, \cdots, p_\omega$ and positive integers $m_1, m_2, \cdots, m_\omega$ as corresponding exponents such that $n = p_1^{m_1} p_2^{m_2} \cdots  p_\omega^{m_\omega}$ where $p_1 < p_2 < \cdots < p_\omega$. Let $M(n)=  (m_1, m_2, \cdots, m_\omega)$  be the sequence of the exponents.
In~\cite{HW1979},  Hardy and Wright used $\Omega(n)$  and $\omega(n)$  to denote the number of prime divisors of $n$ counted with multiplicity and the number of distinct prime factors of $n$, respectively. 
For  example,  $20 = 2 \times 2 \times 5 = 2^2 \times 5^1$ has $\Omega(20) = 3$ and $\omega(20) = 2$. 

%Hardy, G. H. and Wright, E. M. ``The Number of Prime Factors of $n$'' and ``The Normal Order of $\omega(n)$ and %$\Omega(n)$.''  pp. 354-358, 1979.  

Let $M'(n) =  [m_1, m_2, \cdots, m_\omega]$ be the {\em multiset} known as the {\em prime signature} of $n$ where the order does not matter and repetitions are allowed.
For example, $M'(4500 = 2^2\times3^2\times5^3) = [2, 2, 3]$ has the same  prime signature as $M'(33075 = 3^3\times5^2\times7^2) = [3,2,2]$.  
The prime signature $M'(n)$ uniquely determines the structures of $G^H(n)$ and $G^T(n)$ and  play a central role in this work as they partition the $G^H(n)$ and $G^T(n)$   into isomorphism classes and 
are used as the labels of the nodes of   $G^H(n)$ and $G^T(n)$ .

Any ordering of the prime signatures corresponds to an ordering of the isomorphism classes of  $G^H(n)$ and $G^T(n)$   and consequently of their associated graph invariants, such as their order, size, and path counts. 
Two kinds of orderings of prime signatures such as the {\em graded colexicographic} and 
{\em canonical orderings} appear  in the literature and the On-line Encyclopedia of Integer Sequences~\cite{oeis}.
Several integer sequences by prime signatures  have been  studied from the number theory point of view~\cite{AS1972,HW1979}, 
the earliest one of which  dates from 1919~\cite{MacMahon1919}. 
However, some sequences have interpretations different from the graph theory interpretations provided here. 
Most importantly, over twenty new integer sequences of great interest are presented in Section~\ref{s3}.

\section{Graph Theoretic Properties and Invariants of the Divides Relation} % use lowercase except for proper names
\label{s2}
In this section, fourteen graph invariants such as order, size, degree, etc. for the {\em Hasse Diagram} and/or {\em Transitive Closure} graphs are formally defined and investigated. Furthermore, various graph theoretic properties are also determined.

The first graph invariant of interest is the common  {\em order} of  $G^H(n)$ and $G^T(n)$, i.e., the number of nodes, $|V(n)|$. By definition,  this is simply the number of divisors of  $n$. 
\begin{theorem}[Order of $G^H(n)$ and $G^T(n)$]
\label{Tmorder}
\begin{equation}
|V(n)| = |V(M(n))|= \prod_ {m_i \in M(n)}(m_i + 1)
\end{equation}
\end{theorem}
\begin{proof}
Each $p_i^{m_i}$ term contains $m_i + 1$ factors which can  contribute to a divisor of $n$. Thus, the number of divisors of $n$ is $(m_1+1)\times(m_2+1)\times\cdots\times(m_\omega+1)$ by the {\em product rule of counting}.  
\end{proof}
This classic and important integer sequence of $|V(n)|$ in natural order is given in Table~\ref{t01} and listed as  A000005 in~\cite{oeis}.
Table~\ref{t01} lists 14  integer sequences of all forthcoming graph invariants with OEIS number if listed and blank in the OEIS column if not listed.
\begin{table}[hbtp]\vspace*{-3ex}
\caption[]{divides relation graph invariants in natural order} 
\label{t01}
\centering
{\footnotesize
\begin{tabular}{cp{3.6in}c} \hline
\multicolumn{1}{c}{Invariant} &
\multicolumn{1}{c}{Integer sequence for $n = 1,\cdots,50$}  &
\multicolumn{1}{c}{OEIS} \\ \hline \hline
$|V(n)|$ & 1, 2, 2, 3, 2, 4, 2, 4, 3, 4, 2, 6, 2, 4, 4, 5, 2, 6, 2, 6, 4, 4, 2, 8, 3, 4, 4, 6, 2, 8, 2, 6, 4, 4, 4, 9, 2, 4, 4, 8, 2, 8, 2, 6, 6, 4, 2, 10, 3, $\cdots$   & A000005\\ \hline
$|E^H(n)|$ &0, 1, 1, 2, 1, 4, 1, 3, 2, 4, 1, 7, 1, 4, 4, 4, 1, 7, 1, 7, 4, 4, 1, 10, 2, 4, 3, 7, 1, 12, 1, 5, 4, 4, 4, 12, 1, 4, 4, 10, 1, 12, 1, 7, 7, 4, 1, 13, $\cdots$   & A062799\\ \hline
$\Omega(n)$ &0, 1, 1, 2, 1, 2, 1, 3, 2, 2, 1, 3, 1, 2, 2, 4, 1, 3, 1, 3, 2, 2, 1, 4, 2, 2, 3, 3, 1, 3, 1, 5, 2, 2, 2, 4, 1, 2, 2, 4, 1, 3, 1, 3, 3, 2, 1, 5, 2, 3, 2, $\cdots$   & A001222\\ \hline
$\omega(n)$ &0, 1, 1, 1, 1, 2, 1, 1, 1, 2, 1, 2, 1, 2, 2, 1, 1, 2, 1, 2, 2, 2, 1, 2, 1, 2, 1, 2, 1, 3, 1, 1, 2, 2, 2, 2, 1, 2, 2, 2, 1, 3, 1, 2, 2, 2, 1, 2, 1, 2, 2, 2, $\cdots$   & A001221 \\ \hline
$W_v(n)$ & 1, 1, 1, 1, 1, 2, 1, 1, 1, 2, 1, 2, 1, 2, 2, 1, 1, 2, 1, 2, 2, 2, 1, 2, 1, 2, 1, 2, 1, 3, 1, 1, 2, 2, 2, 3, 1, 2, 2, 2, 1, 3, 1, 2, 2, 2, 1, 2, 1, 2, $\cdots$   & A096825\\ \hline
$W_e(n)$ & 1, 1, 1, 1, 1, 2, 1, 1, 1, 2, 1, 3, 1, 2, 2, 1, 1, 3, 1, 3, 2, 2, 1, 3, 1, 2, 1, 3, 1, 6, 1, 1, 2, 2, 2, 4, 1, 2, 2, 3, 1, 6, 1, 3, 3, 2, 1, 3, 1, 3, $\cdots$   & -\\ \hline
$\Delta(n)$ &0, 1, 1, 2, 1, 2, 1, 2, 2, 2, 1, 3, 1, 2, 2, 2, 1, 3, 1, 3, 2, 2, 1, 3, 2, 2, 2, 3, 1, 3, 1, 2, 2, 2, 2, 4, 1, 2, 2, 3, 1, 3, 1, 3, 3, 2, 1, 3, 2, 3, $\cdots$   & -\\ \hline
$|P^H(n)|$ &  	1, 1, 1, 1, 1, 2, 1, 1, 1, 2, 1, 3, 1, 2, 2, 1, 1, 3, 1, 3, 2, 2, 1, 4, 1, 2, 1, 3, 1, 6, 1, 1, 2, 2, 2, 6, 1, 2, 2, 4, 1, 6, 1, 3, 3, 2, 1, 5, 1, 3, 2, 3, $\cdots$   & A008480\\ \hline
$|V_E(n)|$ & 1, 1, 1, 2, 1, 2, 1, 2, 2, 2, 1, 3, 1, 2, 2, 3, 1, 3, 1, 3, 2, 2, 1, 4, 2, 2, 2, 3, 1, 4, 1, 3, 2, 2, 2, 5, 1, 2, 2, 4, 1, 4, 1, 3, 3, 2, 1, 5, 2, 3, $\cdots$   &A038548\\ \hline
$|V_O(n)|$ &0, 1, 1, 1, 1, 2, 1, 2, 1, 2, 1, 3, 1, 2, 2, 2, 1, 3, 1, 3, 2, 2, 1, 4, 1, 2, 2, 3, 1, 4, 1, 3, 2, 2, 2, 4, 1, 2, 2, 4, 1, 4, 1, 3, 3, 2, 1, 5, 1, 3, $\cdots$   & A056924\\ \hline
$|E_E(n)|$ & 0, 1, 1, 1, 1, 2, 1, 2, 1, 2, 1, 4, 1, 2, 2, 2, 1, 4, 1, 4, 2, 2, 1, 5, 1, 2, 2, 4, 1, 6, 1, 3, 2, 2, 2, 6, 1, 2, 2, 5, 1, 6, 1, 4, 4, 2, 1, 7, 1, 4, $\cdots$   & - \\ \hline
$|E_O(n)|$ & 0, 0, 0, 1, 0, 2, 0, 1, 1, 2, 0, 3, 0, 2, 2, 2, 0, 3, 0, 3, 2, 2, 0, 5, 1, 2, 1, 3, 0, 6, 0, 2, 2, 2, 2, 6, 0, 2, 2, 5, 0, 6, 0, 3, 3, 2, 0, 6, 1, 3, $\cdots$   & -\\ \hline
$|E^T(n)|$ & 0, 1, 1, 3, 1, 5, 1, 6, 3, 5, 1, 12, 1, 5, 5, 10, 1, 12, 1, 12, 5, 5, 1, 22, 3, 5, 6, 12, 1, 19, 1, 15, 5, 5, 5, 27, 1, 5, 5, 22, 1, 19, 1, 12, 12, 5, $\cdots$   & - \\ \hline
$|P^T(n)|$ & 1, 1, 1, 2, 1, 3, 1, 4, 2, 3, 1, 8, 1, 3, 3, 8, 1, 8, 1, 8, 3, 3, 1, 20, 2, 3, 4, 8, 1, 13, 1, 16, 3, 3, 3, 26, 1, 3, 3, 20, 1, 13, 1, 8, 8, 3, 1, 48, 2, $\cdots$   & A002033  \\ \hline
\end{tabular} }
\end{table}

%\begin{lemma}[Parity of $|V(n)|$] 
%$|V(n)|$ is odd if and only if  $\forall m_i \in M(n)$ are even.  
%\label{Tmoddo}
%\end{lemma}
%\begin{proof}
%Suppose there exists $m_x$ which is odd. Then the $m_x + 1$ term is even and  thus $|V(n)|$ is even since  $|V(n)|$ is product of $m_i + 1$ terms as in Theorem~\ref{Tmorder}.
%\end{proof}

The next eleven graph invariants of interest are for $G^H(n)$ exclusively. 
The second graph invariant of interest is the {\em size} of $G^H(n)$ which is the cardinality of the arc set $|E^H(n)| = |E^H(M(n))|$. 
A recursive algorithm to compute $|E^H(n)|$ is given in Algorithm~\ref{Ahsize} which utilizes a size fact about the {\em Cartesian product} of two graphs.  
\begin{algorithm}[Size of $G^H(n)$] 
Let $m_i \in M'$ and the multiset, $M = M'(n)$ initially. 
\label{Ahsize}
\begin{equation}
|E^H(M)|=  
\left\{
  \begin{array}{l l}
   |E^H(M - \{m_i\})|\times(m_i+1) + m_i\times|V(M- \{m_i\})| &  \textrm{if }  |M| > 1 \\
    m_1 &  \textrm{if }  |M| = 1
  \end{array} \right.
\end{equation}
\end{algorithm}
\begin{theorem}[Algorithm~\ref{Ahsize} correctly computes $|E^H(n)|$]  
\end{theorem}
\begin{proof}
In~\cite{Harary1972}, a theorem about the size of the {\em Cartesian product} of two graphs is given, i.e.,   the size of a {\em Cartesian product} of two graphs is the size of the first multiplied by the order of the second added to the size of the second multiplied by the order of the first.  
Using this theorem and the fact that $G^H(n)$ is isomorphic to the Cartesian product of paths, 
%( see~\cite{EQ2014}), 
it is clear inductively that the recursive Algorithm~\ref{Ahsize} correctly computes the size of $G^H(n)$. 
\end{proof}
The integer sequence of $|E^H(n)|$ is listed as A062799 with an alternative formula  and described as the {\em inverse M\"{o}bius transform} of the number of distinct prime factors of $n$ in~\cite{oeis}.

For the purpose of illustrating the various concepts that are defined in what follows  $G^H(540)$ is shown in Figure~\ref{f02}.
Note that $540 = 2^23^35$ and that the nodes of  $G^H(540)$  are labeled with the sequence of exponents with respect of the order of  $M(n)$. Each node $v  \in V(n)$ is expressed as a sequence, $M_n(v) = (v_1,\cdots, v_{\omega(n)})$ where $0 \le v_i \le m_i$. 
\begin{definition}[Node as a sequence] If $v \in V(n)$ and $n = p_1^{m_1}p_2^{m_2} \cdots p_\omega^{m_\omega}$, then 
\begin{equation}
v = p_1^{v_1}p_2^{v_2} \cdots p_\omega^{v_\omega} \textrm{ and } M_n(v) = (v_1, v_2, \cdots, v_\omega)   
\end{equation}
\end{definition}
To minimize clutter in Figure~\ref{f02} the sequences $(2,3,1), (2,3,0), \cdots, (0,0,0)$ are written \newline
2 3 1, 2 3 0, $\cdots$, 0 0 0.
\begin{figure}[htb]
\centering
%\resizebox{!}{2.9in}{\includegraphics{f03.eps}}
\includegraphics[scale=.5]{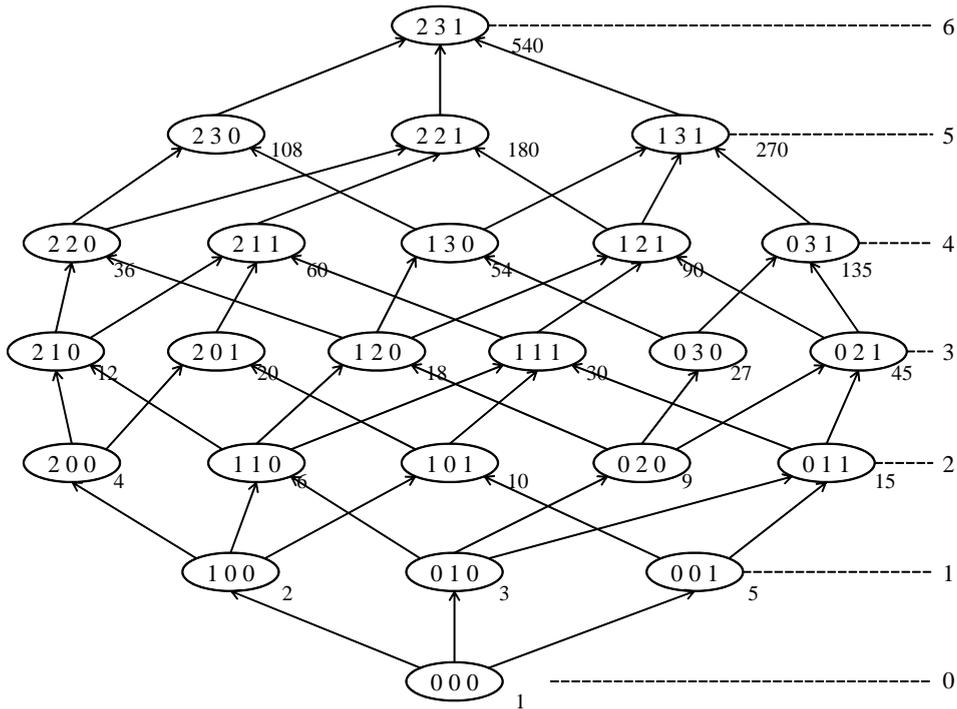}
\caption{\label{f02} $G^H(540) = G^H(M(540)) = G^H((2,3,1))$.}
\end{figure}

Let $V_l(n)$ denote the set of nodes lying in the $l$ level of the  decomposition of $G^H(n)$. For example in Figure~\ref{f02}, $V_5(540) = \{108, 180, 270\}$.  
\begin{lemma}[The sum of the prime signature of a node equals its level]
\label{Lmlevel}
\begin{equation}
V_l(n) = \left\{v \in V(n) \suchthat{\sum_{v_i \in M_n(v)}v_i = l}\right\}
\end{equation}
\end{lemma}
\begin{proof}
If $v \in V(n)$, then $v = n/x$, where $x$ is the product of $\Omega(n) - l$ primes (multiplicities counted) contained in 
$\{p_1, p_2, \cdots p_w\}$. Thus, the nodes in $V_l(n)$ are precisely the nodes with signature sum $\sum_{v_i \in M_n(v)}v_i = l$.
\end{proof}
\begin{observation} Nodes partitioned by their level. 
\begin{equation}
V_{l_1}(n) \cap V_{l_2}(n)  = {\O}   \textrm{ if }  l_1 \not= l_2 \wedge l_1, l_2 \in \{0,..,\Omega(n)\}
\end{equation}
\begin{equation}
 V(n)  = \bigcup_{l \in \{0,..,\Omega(n)\}} V_l(n) 
\end{equation}
\begin{equation}
|V(n)|   =  \sum_{l \in \{0,..,\Omega(n)\}}|V_l(n)| 
\end{equation}
\end{observation}

Let $P(x,y)$ be the set of  paths from  node $x$ to node $y$ in a directed acyclic graph where each path is a sequence of arcs from $x$ to $y$. For example in $G^H(20)$ as shown in Figure~\ref{f01} (b),   
$P(1,20) = \{\langle(1,2),(2,4),(4,20)\rangle, \langle(1,2),(2,10),(10,20)\rangle, \langle(1,5),(5,10),(10,20)\rangle\}$. 
Let $sp(x,y)$ and $lp(x,y)$ be the lengths of  the shortest path and longest path from $x$ to  $y$. 
Let $G(n)$ be a directed acyclic graph with a single source node, $1$ and a single sink node, $n$. 
Let $sp(G(n))$ and $lp(G(n))$ be the lengths of the shortest path and longest path from $1$ to  $n$, respectively. 
For simplicity sake, we shall denote $P^H(n)$ and $P^T(n)$ for  $P(1,n)$ in $G^H(n)$ and $G^T(n)$, respectively.

The height of $G^H(n)$  is the maximum level in the level decomposition  of $G^H(n)$, namely   the number of prime factors.
\begin{theorem}[Height of $G^H(n)$] 
\label{Tmdepth}
\begin{equation}
height(G^H(n)) = sp(G^H(n)) = \sum_ {m_i \in M(n)}m_i  = \Omega(n)
\end{equation}
\end{theorem}
\begin{proof}
Follows directly from  Lemma~\ref{Lmlevel}.
\end{proof}
\begin{corollary}[Length of Paths in $G^H(n)$ and $G^T(n)$]  
\label{Cldepth}
\begin{equation}
sp(G^H(n)) = lp(G^H(n)) = lp(G^T(n)) = \Omega(n)
\end{equation}
\end{corollary}
\begin{proof}
Follows directly from  Lemma~\ref{Lmlevel}.
\end{proof}
Note that $sp(G^T(n)) = 1$ since the arc with a single path, $(1,n) \in P^T(n)$.

\begin{theorem}[Symmetry of $V_l(n)$] 
\label{Tmsymmv}
\begin{equation}
|V_l(n)| = |V_{\Omega(n)-l}(n)|
\end{equation}
\end{theorem}
\begin{proof}
A $1-1$ correspondence $f$ is defined between $V_l(n)$ and $V_{\Omega(n)-l}(n)$. 
Let $v$ be a node in $V_l(n)$ and  
$f$ the function from $V_l(n)$ to  $V_{\Omega(n)-l}(n)$  defined by 
\begin{equation}
f(v) = p_1^{m_1 - v_1} p_2^{m_2 - v_2}\cdots p_\omega^{m_\omega - v_\omega}
\label{ef}
\end{equation}
By Lemma~\ref{Lmlevel}, $f(v)$ is on level $\Omega(n) - l$ and $f$ is clearly $1-1$ into. 
Similarly, the function $g$ from   $V_{\Omega(n)-l}(n)$ to $V_l(n)$  defined by 
\begin{equation}
g(u) = p_1^{m_1 - u_1} p_2^{m_2 - u_2}\cdots p_\omega^{m_\omega - n_\omega}  \textrm{ where  } u \in V_{\Omega(n)-l}(n)
\label{eg}
\end{equation}
is  clearly $1-1$ into with $g(u)$ in $V_l(n)$. Thus, $g$ is $f^{-1}$ and $|V_l(n)| = |V_{\Omega(n)-l}(n)|$. 
\end{proof}

Let $E^H_l(n)$ be the set of arcs from nodes in level $l$ to  level $l+1$ and formally defined in Definition~\ref{dlev}.
\begin{definition}
\label{dlev}
\begin{equation}
E^H_l(n) = \{(a,b) \in E^H(n) | a \in V_l(n)\}
\end{equation}
\end{definition}
For example in  Figure~\ref{f02}, $E^H_0(540) = \{(1,2), (1,3), (1,5)\}$ and \newline
 $E^H_5(540) = \{(108,540),  (180,540),  (270,540)\}$.  
The following is a symmetry property of $E^H(n)$.
\begin{theorem}[Symmetry of $E^H_l(n)$] 
\label{Tmsymme}
\begin{equation}
|E^H_l(n)| = |E^H_{\Omega(n)-l-1}(n)|
\end{equation}
\end{theorem}
\begin{proof}
Let $a \in V_l(n)$ and $b \in V_{l+1}(n)$, and $(a,b)$ be an arc from $V_l(n)$ to  $V_{l+1}(n)$. 
Then, using $f$ in~(\ref{ef}), the function  $F$ defined by $F(a,b) = (f(b),f(a))$ provides a $1-1$ into function from 
$E^H_l(n)$ to $E^H_{\Omega(n)-l-1}(n)$. 
This is seen by noting that 
\begin{eqnarray}
f(b) & = & p_1^{m_1-b_1} p_2^{m_2-b_2} \cdots p_\omega^{m_\omega-b_\omega} \textrm{ is in } V_{\Omega(n) - l - 1}\\
f(a) & = & p_1^{m_1-a_1} p_2^{m_2-a_2} \cdots p_\omega^{m_\omega-a_\omega} \textrm{ is in } V_{\Omega(n) - l}\\
\frac{f(a)}{f(b)} & = & \frac{p_1^{m_1-a_1} p_2^{m_2-a_2} \cdots p_\omega^{m_\omega-a_\omega}}{p_1^{m_1-b_1} p_2^{m_2-b_2}\cdots p_\omega^{m_\omega-b_\omega}} \nonumber \\ 
& = & \frac{p_1^{m_1} p_2^{m_2} \cdots p_\omega^{m_\omega} p_1^{b_1} p_2^{b_2} \cdots p_\omega^{b_\omega}}
{p_1^{m_1} p_2^{m_2} \cdots p_\omega^{m_\omega}p_1^{a_1} p_2^{a_2} \cdots p_\omega^{a_\omega}} = \frac{b}{a} = p   \label{epf1}
\end{eqnarray}
Thus, from~(\ref{epf1}), since $(a,b)$ is an arc, $(f(b),f(a))$ is an arc from $V_{\Omega(n) - l - 1}$ to  $V_{\Omega(n) - l}$.
Therefore, $F$ provides a $1-1$ into function from $E^H_l(n)$ to $E^H_{\Omega(n)-l-1}(n)$. 
Similarly, the function $G$ defined by $G(c,d) = (g(d),g(c))$ is a $1-1$ into function 
from $E^H_{\Omega(n)-l-1}(n)$ to  $E^H_l(n)$ .  
Therefore, $|E^H_l(n)| = |E^H_{\Omega(n)-l-1}(n)|$. 
\end{proof}

All $G^H(n)$ have a single source node, $1$ and a single sink node, $n$. Thus $|V_0(n)| = |V_{\Omega(n)}(n)| = 1$.
There are two other special levels with $\omega(n)$ as their cardinalities. 
\begin{theorem}[Two special levels with $\omega(n)$ nodes] 
\label{Tmuniqp}
\begin{equation}
|V_{\Omega(n)-1}(n)|= |V_1(n)| = \omega(n) 
\end{equation}
\end{theorem}
\begin{proof}
$V_1(n)$ consists of the $\omega(n)$ distinct prime factors of $n$. By Theorem~\ref{Tmsymmv} $|V_1(n)| = |V_{\Omega(n)-1}(n)|  = \omega(n) $. 
\end{proof}

\begin{definition} Width of $G^H(n)$ in terms of nodes
\begin{equation}
W_v(n) = \max_{l \in \{0,..,\Omega(n)\}}|V_l(n)|
\end{equation}
\end{definition}
For example in Figure~\ref{f02}, $W_v(540) = 6$ at level $3$. 
The $W_v(n)$ sequence is listed as A096825, the maximal size of an {\em antichain} in a divisor lattice  in~\cite{oeis}. 
A different width can be defined in terms of arc cardinality in each level as depicted in Figure~\ref{f03}.  
\begin{definition} Width of $G^H(n)$ in terms of arcs
\begin{equation}
W_e(n) = \max_{l \in \{0,..,\Omega(n)-1\}}|E^H_l(n)|
\end{equation}
\end{definition}
\begin{figure}[htb]
\begin{center}
\resizebox{!}{2.4in}{\includegraphics{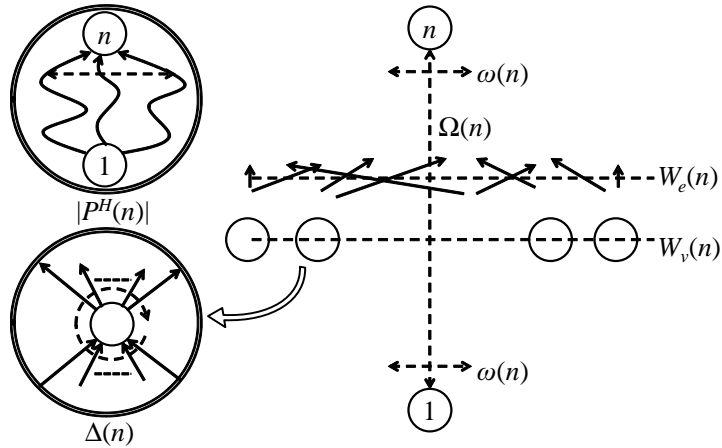}}
\end{center}
\caption{\label{f03} Anatomy of $(n)$.}
\end{figure}
For example in Figure~\ref{f02}, $W_e(540) = 12$ at levels $2$ and $3$. 
The $W_e(n)$ sequence does not appear in~\cite{oeis}.

Since $G^H(n)$ is a digraph, each node, $v$ has an {\em in-degree}, $\Delta^-(v)$, number of incoming arcs and an {\em out-degree}, $\Delta^+(v)$, number of outgoing arcs and the degree of $v$ is defined  $\Delta(v) = \Delta^+(v) + \Delta^-(v)$ . 
\begin{lemma}[Upper bound for indegrees and outdegrees]  
\label{Tmbinout}
For a node $v \in V(n)$, \newline
\[ \Delta^-(v) \le \omega(n), \Delta^+(v) \le \omega(n), \textrm{ and } \Delta(v) \le 2\omega(n)\]
\end{lemma}
\begin{proof}
For the outdegree, each node can add at most one more of each distinct prime to the product. 
For the indegree, the product represented by the node was obtained by adding at most one prime to the product at the level just below. 
\end{proof}

\begin{definition} The degree of the graph $G^H(n)$ denoted , $\Delta(G^H(n))$ is defined by 
\begin{equation}
\Delta(G^H(n)) = \max_{v \in V(n)} \Delta(v)
\end{equation} 
\end{definition}
For example from Figure~\ref{f02}, $\Delta(G^H(540)) = 5$ because the maximum node degree of $G^H(540)$ occurs at $90, 30, 18,$ and $6$. 
The $\Delta(G^H(n))$ or simply $\Delta(n)$ sequence is not listed in~\cite{oeis}. 
The $\Delta(G^H(n))$ can be computed very efficiently as stated in Theorem~\ref{Tdelta} using only $M'(n)$. 
Let $G(n)$ be a sub-multiset of $M'(n)$.
\begin{eqnarray}
G(n) & =&  [m_i \in M'(n) \suchthat m_i > 1] \\
|G(n)| & = & \sum_{m_i \in M(n)}gto(m_i) \textrm{ where } gto(m_i) = 
\left\{
  \begin{array}{l l}
   1 &  \textrm{if }  m_i > 1 \\
   0 &  \textrm{otherwise }
  \end{array} \right.
\end{eqnarray} 
For example of $M'(540) = [2, 3, 1]$, $G(540) = [2, 3]$, and $|G(540)| = 2$. 
\begin{theorem}[Degree of $G^H(n)$]  
\label{Tdelta}
\begin{equation}
\Delta(G^H(n)) = \omega(n) + |G(n)|
\end{equation}
\end{theorem}
\begin{proof}
Consider $v  \in V(n)$ with $M_n(v) = (v_1,\cdots, v_{\omega})$ where $0 \le v_i \le m_i$. 
For a $v_i$ whose $m_i > 1$, $v$ has an incoming arc from a node $u$ whose $M_n(u) = (v_1,\cdots, (u_i = v_i - 1), \cdots, v_{\omega})$ provided $v_i > 0$ and $v$ has an outgoing arc to a node  $w$ whose $M_n(w)= (v_1,\cdots, (w_i = v_i + 1), \cdots, v_{\omega})$ as long as $v_i  < m_i$.
Every element in $G(n)$  contributes $2$ to $\Delta(v)$. 
For a $v_i$ in the $M'(n) - G(n)$ multiset, whose $m_i = 1$, $v$ can have either only the incoming arc from a node $u$ whose $M_n(u) = (v_1,\cdots, (u_i = 0), \cdots, v_{\omega(n)})$ if $v_i = 1$ or the outgoing arc to a node $w$ whose $M_n(w) = [v_1,\cdots, (w_i = 1), \cdots, v_{\omega(n)}]$ if $v_i = 0$. There are $\omega(n) - |G(n)|$ number of such elements, $\le 1$.
Therefore, for every node $v \in V(n)$, $\Delta(v) \le 2 \times |G(n)|+ \omega(n) - |G(n)| = \omega(n) + |G(n)|$. There exists a node $v$ whose $\Delta(v) = \omega(n) + |G(n)|$. One such node is $v$ such that $M_n(v) = (m_1-1, m_2-1,\cdots,m_{\omega(n)}-1)$. 
\end{proof}
For example in Figure~\ref{f02}, in $G^H((2,3,1))$, the node $18$ whose $M_n(18) = (1,2,0)$ has the maximum degree, $5$. 

The next graph invariant of interest is the cardinality of paths, $|P(G^H(n))|$. 
The first 200 integer sequence entries match with those labeled as A008480~\cite{oeis} which is the number of ordered prime factorizations of $n$  with its multinomial coefficient formula given in Theorem~\ref{Tmnumop}~\cite{AS1972,KKW1993}. 
\begin{theorem}[the number of ordered prime factorizations of $n$~\cite{AS1972,KKW1993}]  
\label{Tmnumop}
\begin{equation}
opf(n)= \frac{(\sum_{x \in M(n)} x)!}{ \prod_{x \in M(n)} x!}
\end{equation}
\end{theorem}
While a nice formula has been given in~\cite{AS1972,KKW1993}, a recursive definition is given here where the {\em dynamic programming} technique can be applied to quickly generate the integer sequence.  
\begin{theorem}[Cardinality of $P(G^H(n))$]  
\label{Tmpnumh}
\begin{equation}
|P(G^H(n))|=
\left\{
  \begin{array}{l l}
   \sum\limits_{v \in V_{\Omega(n)-1}(n)}|P(G^H(v))|  &  \textrm{if } \Omega(n) > 1  \\
    1 &  \textrm{if }  \Omega(n) \le  1
  \end{array} \right.
\end{equation}
\end{theorem}
\begin{proof}
All paths in $P(G^H(n))$ must contain exactly one node at level $\Omega(n) - 1$. 
\end{proof}

The next four graph invariants involve the fact that $G^H(n)$ is {\em bipartite} as depicted in Figure~\ref{f04}. 
\begin{theorem}[$G^H(n)$ is bipartite]  
\label{Tmbipart}
\end{theorem}
\begin{proof}
Arcs join only even level nodes to odd level nodes and vice versa. 
Thus, the nodes at even and odd levels form a bipartition of $V(n)$. 
\end{proof}
\begin{figure}[htb]
\begin{center}
\begin{tabular}{cc}
\resizebox{!}{1.8in}{\includegraphics{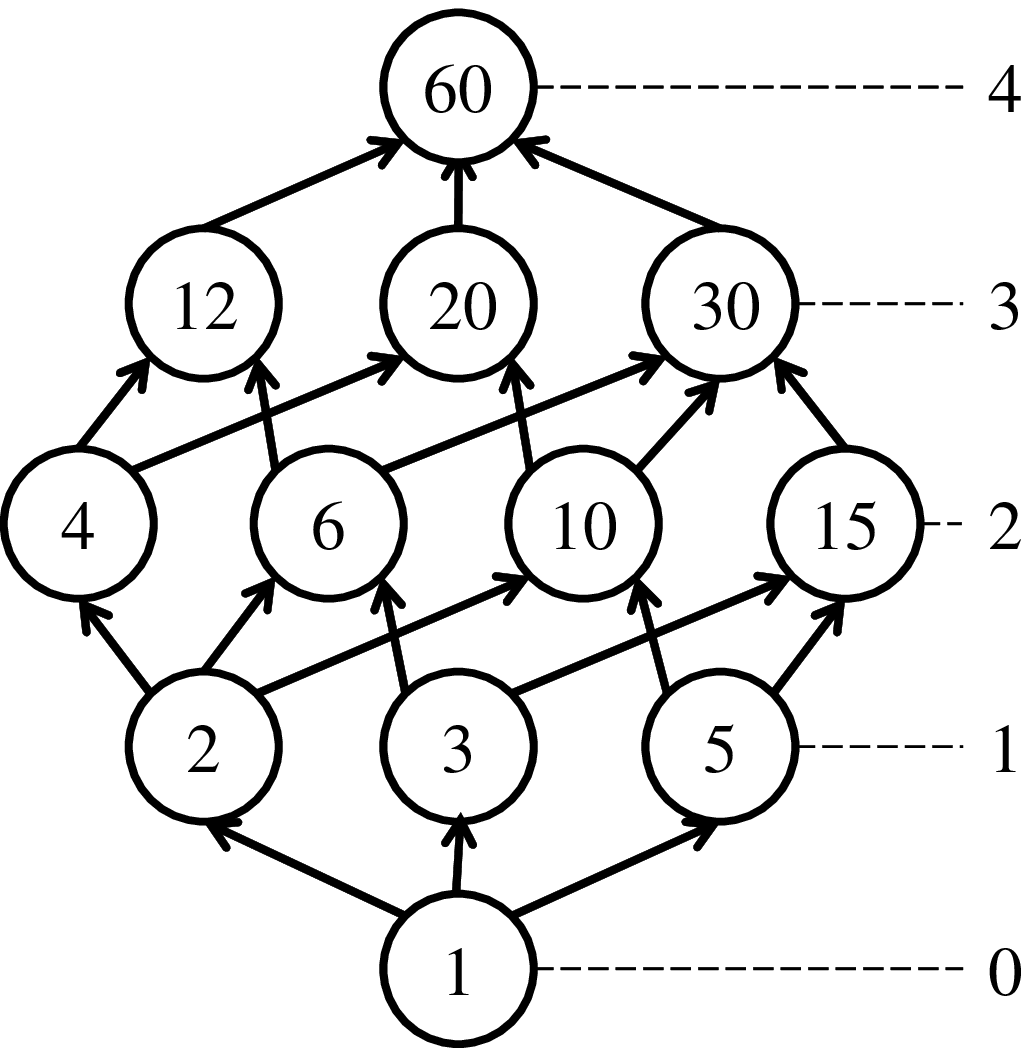}}
&
\resizebox{!}{1.8in}{\includegraphics{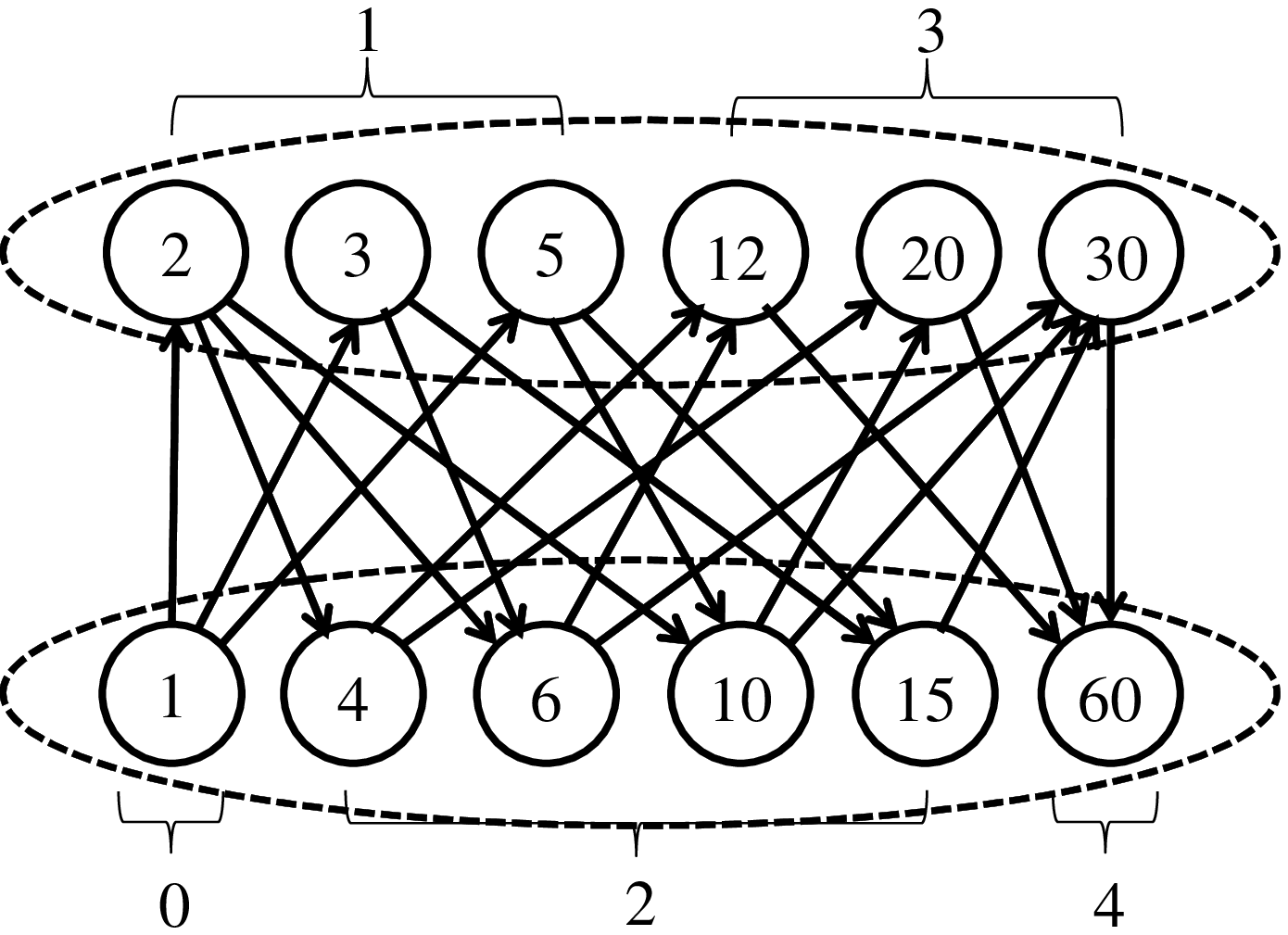}}\\
(a) Hasse diagram $G^H(60)$ & (b) $G^H(60)$ shown as a bipartite graph 
\end{tabular}
\end{center}
\caption{\label{f04} $G^H(60)$}
%= G^H(2^2 \times 3 \times  5) =  G^H(p^2 q r)$.}
\end{figure}

\begin{definition}
\label{Defoddv}
\begin{eqnarray}
V_E(n) & = & \{v \in V(n) \mid \sum_{m_i \in M_n(v)}m_i = even\}  \\
V_O(n) & = & \{v \in V(n) \mid \sum_{m_i \in M_n(v)}m_i = odd\}  
\end{eqnarray} 
\end{definition}
The integer sequence of the cardinality of $V_E$  matches with  A038548 which is the   number of divisors of $n$ that are at most $\sqrt{n}$~\cite{oeis,Andrews2004}. 
The integer sequence of $|V_O|$ also appears as A056924, described  as the number of divisors of $n$ that are smaller than $\sqrt{n}$~\cite{oeis,Andrews2004}. 
\begin{theorem}[Cardinality of $V_O(n)$]  
\label{TmnumOv}
\begin{equation}
|V_O(n)| = \left \lfloor \frac{|V(n)|}{2}\right \rfloor
\end{equation}
\end{theorem}
\begin{proof}
The proof is by induction. For the base case $\omega = 1$, each divisor has a single exponent, i.e., $v_i \in \{p_1^0,p_1^1,\cdots,p_1^{m_1}\}$. Clearly, $|V_O| = \left \lfloor \frac{|V(n)|}{2}\right \rfloor$. 
For the inductive step $\omega + 1$, let $M_{\omega + 1}$ be $M_{\omega}$ with $m_{\omega + 1}$ appended.
$V_O(M_{\omega+1})$ is the union of the cartesian product of $V_O(M_{\omega})$ and $V_E(m_{\omega + 1})$ 
together with  the cartesian product of $V_E(M_{\omega})$ and $V_O(m_{\omega + 1})$, thus 
\begin{equation}
|V_O(M_{\omega + 1})| = |V_O(M_{\omega})|\times|V_E(m_{\omega + 1})| + |V_E(M_{\omega})|\times|V_O(m_{\omega + 1})|
\end{equation}
There are four cases depending on the parities of $|V(M_{\omega})|$ and $m_{\omega + 1}$. 
The following uses Theorem~\ref{Tmorder} and Definition~\ref{Defoddv}. 
\newline
If $|V(M_{\omega})|$ is odd and $m_{\omega + 1}$ is odd, 
\begin{eqnarray}
|V_O(M_{\omega + 1})| & = &  \frac{|V(M_{\omega})|-1}{2} \times \frac{m_{\omega + 1}+1}{2} +  \frac{|V(M_{\omega})|+1}{2} \times \frac{m_{\omega + 1}+1}{2}  \nonumber \\
 & = &  \frac{|V(M_{\omega})|(m_{\omega + 1}+1)-(m_{\omega + 1}+1)}{4}    + \frac{|V(M_{\omega})|(m_{\omega + 1}+1)+(m_{\omega + 1}+1)}{4}  \nonumber \\
 & = &  \frac{|V(M_{\omega + 1})|-(m_{\omega + 1}+1) + |V(M_{\omega + 1})|+(m_{\omega + 1}+1)}{4} 
%=   \frac{|V(M_{\omega + 1})|}{2} 
= \left \lfloor \frac{|V(M_{\omega + 1})|}{2}\right \rfloor \nonumber 
\end{eqnarray}
If $|V(M_{\omega})|$ is odd and $m_{\omega + 1}$ is even, 
\begin{eqnarray}
|V_O(M_{\omega + 1})| & = &  \frac{|V(M_{\omega})|-1}{2} \times \frac{m_{\omega + 1}+2}{2} +  \frac{|V(M_{\omega})|+1}{2} \times \frac{m_{\omega + 1}}{2}  \nonumber \\
 & = &  \frac{|V(M_{\omega})|(m_{\omega + 1})-(m_{\omega + 1} + 2)}{4}    + \frac{|V(M_{\omega})|(m_{\omega + 1}+2)+m_{\omega + 1}}{4}  \nonumber \\
 & = &  \frac{|V(M_{\omega})|(2m_{\omega + 1}+2)-(m_{\omega + 1} + 2)+m_{\omega + 1}}{4} 
=   \frac{|V(M_{\omega + 1})|-1}{2} = \left \lfloor \frac{|V(M_{\omega + 1})|}{2}\right \rfloor  \nonumber 
\end{eqnarray}
If $|V(M_{\omega})|$ is even and $m_{\omega + 1}$ is odd, 
\begin{eqnarray}
|V_O(M_{\omega + 1})| & = &  \frac{|V(M_{\omega})|}{2} \times \frac{m_{\omega + 1}+1}{2} +  \frac{|V(M_{\omega})|}{2} \times \frac{m_{\omega + 1}+1}{2}  \nonumber \\
 & = &  \frac{|V(M_{\omega})|(m_{\omega + 1}+1)}{4}    + \frac{|V(M_{\omega})|(m_{\omega + 1}+1)}{4} =   \frac{|V(M_{\omega + 1})|}{2}  = \left \lfloor \frac{|V(M_{\omega + 1})|}{2}\right \rfloor  \nonumber 
\end{eqnarray}
If $|V(M_{\omega})|$ is even and $m_{\omega + 1}$ is even, 
\begin{eqnarray}
|V_O(M_{\omega + 1})| & = &  \frac{|V(M_{\omega})|}{2} \times \frac{m_{\omega + 1}+2}{2} +  \frac{|V(M_{\omega})|}{2} \times \frac{m_{\omega + 1}}{2}  \nonumber \\
 & = &  \frac{|V(M_{\omega})|(2m_{\omega + 1}+2)}{4}  =   \frac{|V(M_{\omega + 1})|}{2}  = \left \lfloor \frac{|V(M_{\omega + 1})|}{2}\right \rfloor  \nonumber 
\end{eqnarray}
Therefore, $|V_O(M_{\omega + 1})| = \left \lfloor \frac{|V(M_{\omega + 1})|}{2}\right \rfloor$ in all four cases. 
\end{proof}
\begin{corollary}[Cardinality of $V_E(n)$]  
\label{TmnumEv}
\begin{equation}
|V_E(n)| =|V(n)| - |V_O(n)|  =|V(n)| - \left \lfloor |V(n)|/2\right \rfloor
\end{equation}
\end{corollary}
\begin{proof}
Since $V_E(n)$ and  $V_O(n)$ partition  $V(n)$,  $|V_E(n)| =|V(n)| - |V_O(n)|$.
\end{proof}

Similarly as with $V(n)$, $E(n)$ is bipartite as follows. 
\begin{definition}
\begin{eqnarray}
E_E(n) & = & \{ (a,b) \in E^H(n) \mid \sum_{m_i \in M(a)}m_i = even\}  \\
E_O(n) & = & \{ (a,b) \in E^H(n) \mid \sum_{m_i \in M(a)}m_i = odd\}  
\end{eqnarray} 
\end{definition}
Surprisingly, the integer sequences of $|E_E(n)|$ and $|E_O(n)|$ are not listed in~\cite{oeis}. 
\begin{theorem}[Cardinality of $E_O(n)$]  
\label{TmnumSe}
\begin{equation}
|E_O(n)| = \left \lfloor \frac{|E^H(n)|}{2}\right \rfloor
\end{equation}
\end{theorem}
\begin{proof}
An inductive proof, similar to the proof for the node parity decomposition in Theorem~\ref{TmnumOv}, can be applied using the cartesian product of two graphs~(\ref{eec}) and the arc parity decomposition~(\ref{eeo}). 
\begin{equation}
|E^H(M_{\omega + 1})|  =   |E^H(M_{\omega})| \times |V(m_{\omega + 1})| +  |V(M_{\omega})| \times |E^H(m_{\omega + 1})|
\label{eec} 
\end{equation}
\begin{eqnarray}
|E_O(M_{\omega + 1})| & = &  |E_O(M_{\omega})| \times |V_E(m_{\omega + 1})|  +  |E_E(M_{\omega})| \times |V_O(m_{\omega + 1})| 
\label{eeo} \\
&& +  |V_O(M_{\omega})| \times |E_E(m_{\omega + 1})| +  |V_E(M_{\omega})| \times |E_O(m_{\omega + 1})| \nonumber
\end{eqnarray}
$|E_O(M_{\omega + 1})| = \left \lfloor |E^H(M_{\omega + 1})|/2 \right \rfloor$ 
in all eight cases bsed on parities of $|E^H(M_{\omega})|$,  $|V(M_{\omega})|$, and $m_{\omega + 1}$. 
\end{proof}
\begin{corollary}[Cardinality of $E_E(n)$]  
\label{TmnumEe}
\begin{equation}
|E_E(n)| =|E^H(n)| - |E_O(n)| =|E^H(n)| - \left \lfloor \frac{|E^H(n)|}{2}\right \rfloor
\end{equation}
\end{corollary}
\begin{proof}
Since $E_E(n)$ and  $E_O(n)$ partition $E^H(n)$,  $|E_E(n)| =|E^H(n)| - |E_O(n)|$.
\end{proof}

The last two graph invariants of Table~\ref{t01} are exclusive to the transitive closure, $G^T(n)$, namely the  size and the number of paths in $G^T(n)$. 
Also surprisingly, the sequence for the size of $G^T(n)$ is not listed in~\cite{oeis}. 
\begin{theorem}[Size of $G^T(n)$]
\label{TmsizeT}
\begin{equation}
|E^T(n)| =  \sum_ {v \in V(n)}(|V(v)| - 1)
\end{equation}
\end{theorem}
\begin{proof}
The number of incoming arcs to node $v$ is the number of divisors of $v$ that are less than $v$ itself. Thus the indegree of $v$ is $|V(v)| - 1$ and the sum of the indegrees of all nodes in $G^T(n)$ is the size of $G^T(n)$. 
\end{proof}

\begin{theorem}[Cardinality of $P(G^T(n))$] 
\label{Tmpnumt}
\begin{equation}
|P(G^T(n))|= 
\left\{
  \begin{array}{l l}
   \sum\limits_{v \in V(n)-\{n\}}|P(G^T(v))|  &  \textrm{if } \Omega(n) > 1  \\
    1 &  \textrm{if }  \Omega(n) \le  1
  \end{array} \right.
\end{equation}
\end{theorem}
\begin{proof}
Let $P(G^T(v))$ be the set of all paths from $1$ to $v$ where $v \ne n$. The addition of the arc $(v,n)$ to each path in $P(G^T(v))$ yields a path from $1$ to $n$. Thus, summing over all $v \in V(n) - \{n\}$ is equal to $|P(G^T(n))|$.
\end{proof}
The integer sequence of $|P(G^T(n))|$ matches with 
A002033~\cite{oeis} and described as the number of {\em perfect partitions} of $n$~\cite{oeis,Comtet1974}. 
Thus, the interpretation as the number of paths from $1$ to $n$ is part of the original contributions of this work.

\section{Graph Invariant Integer Sequences ordered by Prime Signature}
\label{s3}
The set of positive integers $> 1$ is partitioned by their prime signatures as exemplified in Table~\ref{t02}. 
\begin{table}[b]\vspace*{-3ex}
\caption[]{Partitions of integers ($> 1$) by prime signature congruency.} 
\label{t02}
\centering
{\footnotesize
\begin{tabular}{p{0.32in}p{3.4in}p{0.9in}} \hline
\multicolumn{1}{c}{$M$ / $S$} &
\multicolumn{1}{c}{Integer sequence for $n = 1,\cdots,20$}  &
\multicolumn{1}{c}{OEIS} \\ \hline \hline
(1) & 2, 3, 5, 7, 11, 13, 17, 19, 23, 29, 31, 37, 41, 43, 47, 53, 59, 61, 67, 71, $\cdots$   & A000040 \newline  (Primes)\\ \hline
(2) & 4, 9, 25, 49, 121, 169, 289, 361, 529, 841, 961, 1369, 1681, 1849, 2209, 2809, 3481, 3721, 4489, 5041, $\cdots$   & A001248 \newline(Squared prime)\\ \hline
(1,1) & 6, 10, 14, 15, 21, 22, 26, 33, 34, 35, 38, 39, 46, 51, 55, 57, 58, 62, 65, 69, $\cdots$   & A006881 \\ \hline
(3) & 8, 27, 125, 343, 1331, 2197, 4913, 6859, 12167, 24389, 29791, 50653, 68921, 79507, 103823, 148877, 205379, 226981, 300763, 357911, $\cdots$   & A030078\newline (Cubed prime)\\ \hline
(2,1)  & 12, 18, 20, 28, 44, 45, 50, 52, 63, 68, 75, 76, 92, 98, 99, 116, 117, 124, 147, 148, $\cdots$   & A054753\\ \hline
(1,1,1) &30, 42, 66, 70, 78, 102, 105, 110, 114, 130, 138, 154, 165, 170, 174, 182, 186, 190, 195, 222, $\cdots$   & A007304 \\ \hline
\multicolumn{1}{c}{$\vdots$} &
\multicolumn{1}{c}{$\vdots$}  &
\multicolumn{1}{c}{$\vdots$} \\
\end{tabular} }
\end{table}

\begin{definition}
$n_x$ and $n_y$ are {\em prime signature congruent}  iff  $M(n_x) =  M(n_y)$. 
\end{definition}

Let $S(n)$ be a representative sequence of the {\em prime signature} $M(n)$ written in descending order. 
More formally, $S(n)  = (s_1, s_2, \cdots, s_\omega)$ is the permutation of the multiset, $M(n) = [m_1, m_2, \cdots, m_\omega]$  such that $s_1 \ge s_2 \ge \cdots \ge s_\omega$.  
For example, $S(4500) = S(33075) = (3,2,2)$ because $M(4500 = 2^2\times3^2\times5^3) = [2, 2, 3]$ has the same  prime signature as $M(33075 = 3^3\times5^2\times7^2) = [3,2,2]$.  

Albeit there are numerous ways of ordering  $S$, the set of all $S(n)$, two particular orderings such as the {\em graded colexicographic} and {\em canonical} orders of $S$ appear in the literature~\cite{HW1979,AS1972}. First in the {\em graded colexicographic order}, $S$ are first grouped by $\Omega(S)$ and then by $\omega(S)$ in ascending order. Finally, the reverse lexicographic order is applied to the sub-group. It is closely related to the {\em graded reflected colexicographic order} used and denoted as $\pi$  in~\cite{AS1972} . 
Let $LI(S)$  denote the least integer of a prime signature  in the graded (reflected or not) colexicographic order. This sequence is listed as A036035 in~\cite{oeis}. 
\begin{figure}[htb]
\begin{center}
{\small
\begin{tabular}{ llllll }
1 (0) & 2 (1) & 3 (2) & 5 (3) & 8 (4) & 13 (5) \\
 &  & 4 (1,1) & 6 (2,1) & 9 (3,1) & 14 (4,1)  \\
  &  &  & 7 (1,1,1)  & 10 (2,2) & 15 (3,2)  \\
  &  &  &  & 11 (2,1,1) & 16 (3,1,1)    \\
  &  &  &  & 12 (1,1,1,1)  & 17  (2,2,1)   \\
  &  &  &  &  & 18 (2,1,1,1)    \\
  &  &  &  &  & 19 (1,1,1,1,1) \\
 \end{tabular}}\\
{\small 
\begin{tabular}{ l l l l l l ll }
Index & Graded Colexicographic & Canonical\\
20 & (6) & (6) \\
21 & (5,1) & (5,1)  \\
22 & (4,2) & (4,2) \\
23 & (3,3) & (4,1,1) \\
24 & (4,1,1) & (3,3)  \\
25 & (3,2,1) & (3,2,1)  \\
26 &  (2,2,2) &  (3,1,1,1) \\
27 & (3,1,1,1) & (2,2,2) \\
28 & (2,2,1,1) & (2,2,1,1) \\
29 & (2,1,1,1,1) & (2,1,1,1,1) \\
30 & (1,1,1,1,1,1) & (1,1,1,1,1,1)
 \end{tabular}}
\end{center}
\caption{\label{f05} First 30 prime signatures in colexicographic  and canonical orders.}
\end{figure}

Next, the {\em canonical order},  also known as  the {\em graded reverse lexicographic order}, is often used to order the partitions~\cite{HW1979}. It  first groups prime signatures by $\Omega(S)$ and then uses the reverse lexicographic order. Although this order is  identical to the {\em graded colexicographic} order for the first 22 prime signatures, they clearly differ at 23, 24, 26, 27, etc., as seen in Figure~\ref{f05}. 
The integer sequence of the least integer, $LI(S)$  in canonical order is listed as  the Canonical partition sequence  encoded by prime factorization (A063008) in~\cite{oeis}.

\begin{figure}[htb]
\begin{center}
\resizebox{5.1in}{!}{\includegraphics{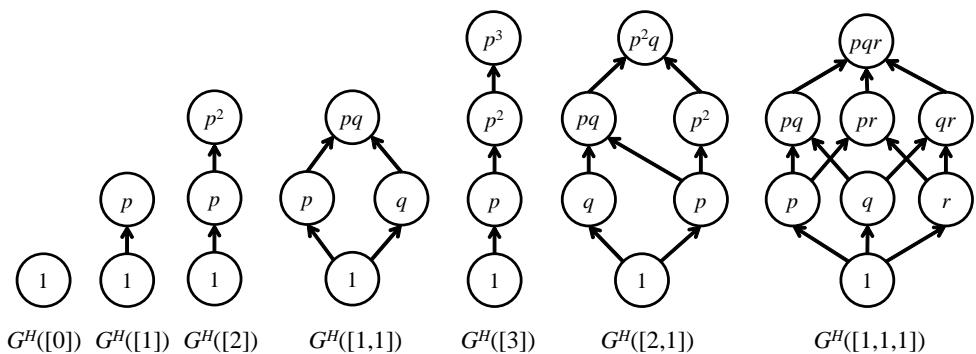}}\\
\end{center}
\caption{\label{f06} First seven Hasse diagrams ordered by prime signatures.}
\end{figure}
The $S(n)$ determine the structure of  $G^H(S(n))$ and $G^T(S(n))$ as shown in Figure~\ref{f06} with the first few simple {\em Hasse diagrams}. 
%\begin{theorem}[Isomorphism]
%\label{TmIsomo}
%For two different positive integers, $n_x$ and $n_y$, \newline 
%\[G^H(n_x) \equiv  G^H(n_y) \textrm{ and } G^T(n_x) \equiv  G^T(n_y) \textrm{ iff } M(n_x) =  M(n_y) \textrm{ or } S(n_x) =  S(n_y).\]
%\end{theorem}
%\begin{proof}
%(A formal proof was done !!!!!! to be included!!!)  
% blah blah blah blah blah blah and 
%thus blah blah blah blah. 
%\end{proof}
%By Theorem~\ref{TmIsomo}, 
All integer sequences of graph invariants in natural order in Table~\ref{t01} can be ordered in the graded colexicographic order (Table~\ref{t04}) and the canonical order (Table~\ref{t05}). 
However, very little  has been investigated concerning  these sequences since most of them are in fact new. 
In~\cite{AS1972}, Abramowitz and Stegun labeled $\Omega(S)$, $\omega(S)$, and $|P^H(S)|$ in the graded colexicographic order as $n$, $m$, and $M_1$, respectively. Only these three graph invariants  and the number of divisors, $|V(S)|$ are found  in~\cite{oeis} for the graded colexicographic order. Only  $|P^H(S)|$ is found  in~\cite{oeis} for the canonical order.

%\begin{table}[b]\vspace*{-3ex}
%\caption[]{Colexico index of  ($n = 1,\cdots,100$).} 
%\label{t03}
%\centering
%{\footnotesize
%\begin{tabular}{p{4.5in}c} \hline
%\multicolumn{1}{c}{Integer sequence for $n = 1,\cdots,100$}  &
%\multicolumn{1}{c}{OEIS} \\ \hline \hline
%1, 2, 2, 3, 2, 4, 2, 5, 3, 4, 2, 6, 2, 4, 4, 8, 2, 5, 2, 6, 4, 4, 2, 9, 3, 4, 5, 6, 2, 7, 2, 13, 4, 4, 4, 10, 2, 4, 4, 9, 2, 7, 2, 6, 6, 4, 2, 14, 3, 5, 4, 6, 2, 8, 4, 9, 4, 4, 2, 11, 2, 4, 6, 20, 4, 7, 2, 6, 4, 7, 2, 15, 2, 4, 5, 6, 4, 7, 2, 14, 8, 4, 2, 11, 4, 4, 4, 9, 2, 10, 4, 6, 4, 4, 4, 21, 2, 5, 6, 10, $\cdots$   & - \\ \hline
%\end{tabular} }
%\end{table}

\section{Conclusion}
\label{s4}
In this article, fourteen graph invariants were investigated for two classic graphs, the {\em Hasse diagram}, $G^H(n)$ and its {\em transitive closure}, $G^T(n)$.  
Integer sequences with their first two hundred entries in  natural order by $n$ are computed and  compared to existing sequences in the On-Line Encyclopedia of Integer Sequences. 
Five new integer sequences in natural order, shown in Table~\ref{t01} were discovered, i.e., not found in~\cite{oeis}. 

New interpretations based on graph theory are provided for sequences found in ~\cite{oeis}.    
Ten (Table~\ref{t04}) and  thirteen (Table~\ref{t05}) new  integer sequences were discovered  for the  graded colexicographic and canonical orders, respectively. 

Here are some  intriguing  conjectures stated as open problems.

\begin{conjecture}[Cardinality of disjoint paths]  Let $P'(G^H(n))$ be the set of {\em disjoint paths}. $|P'(G^H(n))|=\omega(n)$?
\label{Tmdpnum}
\end{conjecture}
%\begin{proof}
%$P'(G^H(n))$ has exactly one disjoint path containing each node at level $\Omega(n) - 1$. 
%Hence, $|P'(G^H(n))|=\omega(n) $ by Theorem~\ref{Tmuniqp}. 
%\end{proof}

\begin{conjecture}[Node width at middle level]   $W_v(n) = |V_{\lceil \Omega(n)/2\rceil}(n)|$?
\label{Conj5}
\end{conjecture}

\begin{conjecture}[Relationship between widths by nodes and arcs]  There always exists a  level $l$ such that 
if $|V_l(n)| = W_v(n)$, then  $|E^H_l(n)| = W_e(n)$.  
\begin{equation}
\argmax_{l \in \{0,..,\Omega(n)-1\}}|E^H_l(n)| = \argmax_{l \in \{0,..,\Omega(n)-1\}}|V_l(n)|?
\end{equation}
\label{Conj6}
\end{conjecture}

Other future work includes finding either a closed and/or a simpler recursive formula for the cardinality of $P(G^T(n))$  in Theorem~\ref{Tmpnumt}.
Note that entries for $|P^T(S)|$  in Table~\ref{t04} and~\ref{t05} are less than 50 as computing  $|P^T(S)|$ by  Theorem~\ref{Tmpnumt} took too long time.

\oneappendix % use \appendix if you have more than one appendix
\section{Integer Sequences by Prime Signatures}
\begin{table}[btp]\vspace*{-3ex}
\caption[]{divides relation graph invariants in graded colexicographic order} 
\label{t04}
\centering
{\footnotesize
\begin{tabular}{cp{3.7in}p{0.4in}} \hline
\multicolumn{1}{c}{Invariant} &
\multicolumn{1}{c}{Integer sequence for $S = [0],\cdots,[4,4]$}  &
\multicolumn{1}{c}{OEIS} \\ \hline \hline
$LI(S)$ &  
1, 2, 4, 6, 8, 12, 30, 16, 24, 36, 60, 210, 32, 48, 72, 120, 180, 420, 2310, 64, 96, 144, 216, 240, 360, 900, 840, 1260, 4620, 30030, 128, 192, 288, 432, 480, 720, 1080, 1800, 1680, 2520, 6300, 9240, 13860, 60060, 510510, 256, 384, 576, 864, 1296, $\cdots$   & A036035 \\ \hline
$|V(S)|$ &1, 2, 3, 4, 4, 6, 8, 5, 8, 9, 12, 16, 6, 10, 12, 16, 18, 24, 32, 7, 12, 15, 16, 20, 24, 27, 32, 36, 48, 64, 8, 14, 18, 20, 24, 30, 32, 36, 40, 48, 54, 64, 72, 96, 128, 9, 16, 21, 24, 25, $\cdots$   & A074139\\ \hline
$|E^H(S)|$ & 0, 1, 2, 4, 3, 7, 12, 4, 10, 12, 20, 32, 5, 13, 17, 28, 33, 52, 80, 6, 16, 22, 24, 36, 46, 54, 72, 84, 128, 192, 7, 19, 27, 31, 44, 59, 64, 75, 92, 116, 135, 176, 204, 304, 448, 8, 22, 32, 38, 40,  $\cdots$   & -\\ \hline
$\Omega(S)$ &0, 1, 2, 2, 3, 3, 3, 4, 4, 4, 4, 4, 5, 5, 5, 5, 5, 5, 5, 6, 6, 6, 6, 6, 6, 6, 6, 6, 6, 6, 7, 7, 7, 7, 7, 7, 7, 7, 7, 7, 7, 7, 7, 7, 7, 8, 8, 8, 8, 8,  $\cdots$   &A036042\\ \hline
$\omega(S)$ &0, 1, 1, 2, 1, 2, 3, 1, 2, 2, 3, 4, 1, 2, 2, 3, 3, 4, 5, 1, 2, 2, 2, 3, 3, 3, 4, 4, 5, 6, 1, 2, 2, 2, 3, 3, 3, 3, 4, 4, 4, 5, 5, 6, 7, 1, 2, 2, 2, 2, $\cdots$   & A036043 \\ \hline
$W_v(S)$ & 1, 1, 1, 2, 1, 2, 3, 1, 2, 3, 4, 6, 1, 2, 3, 4, 5, 7, 10, 1, 2, 3, 4, 4, 6, 7, 8, 10, 14, 20, 1, 2, 3, 4, 4, 6, 7, 8, 8, 11, 13, 15, 18, 25, 35, 1, 2, 3, 4, 5, $\cdots$   & -\\ \hline
$W_e(S)$ & 0, 1, 1, 2, 1, 3, 6, 1, 3, 4, 7, 12, 1, 3, 5, 8, 11, 18, 30, 1, 3, 5, 6, 8, 12, 15, 19, 24, 38, 60, 1, 3, 5, 7, 8, 13, 16, 19, 20, 30, 37, 46, 58, 90, 140, 1, 3, 5, 7, 8, $\cdots$   & -\\ \hline
$\Delta(S)$ & 0, 1, 2, 2, 2, 3, 3, 2, 3, 4, 4, 4, 2, 3, 4, 4, 5, 5, 5, 2, 3, 4, 4, 4, 5, 6, 5, 6, 6, 6, 2, 3, 4, 4, 4, 5, 5, 6, 5, 6, 7, 6, 7, 7, 7, 2, 3, 4, 4, 4, $\cdots$   & -\\ \hline
$|P^H(S)|$ & 1, 1, 1, 2, 1, 3, 6, 1, 4, 6, 12, 24, 1, 5, 10, 20, 30, 60, 120, 1, 6, 15, 20, 30, 60, 90, 120, 180, 360, 720, 1, 7, 21, 35, 42, 105, 140, 210, 210, 420, 630, 840, 1260, 2520, 5040, 1, 8, 28, 56, 70,  $\cdots$   & A036038\\ \hline
$|V_E(S)|$ &1, 1, 2, 2, 2, 3, 4, 3, 4, 5, 6, 8, 3, 5, 6, 8, 9, 12, 16, 4, 6, 8, 8, 10, 12, 14, 16, 18, 24, 32, 4, 7, 9, 10, 12, 15, 16, 18, 20, 24, 27, 32, 36, 48, 64, 5, 8, 11, 12, 13, $\cdots$   & - \\ \hline
$|V_O(S)|$ & 0, 1, 1, 2, 2, 3, 4, 2, 4, 4, 6, 8, 3, 5, 6, 8, 9, 12, 16, 3, 6, 7, 8, 10, 12, 13, 16, 18, 24, 32, 4, 7, 9, 10, 12, 15, 16, 18, 20, 24, 27, 32, 36, 48, 64, 4, 8, 10, 12, 12, $\cdots$   & -\\ \hline
$|E_E(S)|$ & 0, 1, 1, 2, 2, 4, 6, 2, 5, 6, 10, 16, 3, 7, 9, 14, 17, 26, 40, 3, 8, 11, 12, 18, 23, 27, 36, 42, 64, 96, 4, 10, 14, 16, 22, 30, 32, 38, 46, 58, 68, 88, 102, 152, 224, 4, 11, 16, 19, 20,   $\cdots$   & -\\ \hline
$|E_O(S)|$ &0, 0, 1, 2, 1, 3, 6, 2, 5, 6, 10, 16, 2, 6, 8, 14, 16, 26, 40, 3, 8, 11, 12, 18, 23, 27, 36, 42, 64, 96, 3, 9, 13, 15, 22, 29, 32, 37, 46, 58, 67, 88, 102, 152, 224, 4, 11, 16, 19, 20,   $\cdots$   & -\\ \hline

$|E^T(S)|$ & 0, 1, 3, 5, 6, 12, 19, 10, 22, 27, 42, 65, 15, 35, 48, 74, 90, 138, 211, 21, 51, 75, 84, 115, 156, 189, 238, 288, 438, 665, 28, 70, 108, 130, 165, 240, 268, 324, 365, 492, 594, 746, 900, 1362, 2059, 36, 92, 147, 186, 200,  $\cdots$   & - \\ \hline
$|P^T(S)|$ & 1, 1, 2, 3, 4, 8, 13, 8, 20, 26, 44, 75, 16, 48, 76, 132, 176, 308, 541, 32, 112, 208, 252, 368, 604, 818, 1076, 1460, 2612, 4683, 64, 256, 544, 768, 976, 1888, 2316, 3172, 3408, 5740, 7880, 10404, 14300, 25988, $\cdots$   & -  \\ \hline
\end{tabular} }
\end{table}

\begin{table}[btp]\vspace*{-3ex}
\caption[]{divides relation graph invariants in canonical order} 
\label{t05}
\centering
{\footnotesize
\begin{tabular}{cp{3.7in}p{0.4in}} \hline
\multicolumn{1}{c}{Invariant} &
\multicolumn{1}{c}{Integer sequence for $S = [0],\cdots, [5,3]$}  &
\multicolumn{1}{c}{OEIS} \\ \hline \hline

$LI(S)$ &1, 2, 4, 6, 8, 12, 30, 16, 24, 36, 60, 210, 32, 48, 72, 120, 180, 420, 2310, 64, 96, 144, 240, 216, 360, 840, 900, 1260, 4620, 30030, 128, 192, 288, 480, 432, 720, 1680, 1080, 1800, 2520, 9240, 6300, 13860, 60060, 510510, 256, 384, 576, 960, 864, $\cdots$   & A063008 \\ \hline
$|V(S)|$ & 1, 2, 3, 4, 4, 6, 8, 5, 8, 9, 12, 16, 6, 10, 12, 16, 18, 24, 32, 7, 12, 15, 20, 16, 24, 32, 27, 36, 48, 64, 8, 14, 18, 24, 20, 30, 40, 32, 36, 48, 64, 54, 72, 96, 128, 9, 16, 21, 28, 24, $\cdots$   & - \\ \hline
$|E^H(S)|$ & 0, 1, 2, 4, 3, 7, 12, 4, 10, 12, 20, 32, 5, 13, 17, 28, 33, 52, 80, 6, 16, 22, 36, 24, 46, 72, 54, 84, 128, 192, 7, 19, 27, 44, 31, 59, 92, 64, 75, 116, 176, 135, 204, 304, 448, 8, 22, 32, 52, 38,  $\cdots$   & -\\ \hline
$\Omega(S)$ & 0, 1, 2, 2, 3, 3, 3, 4, 4, 4, 4, 4, 5, 5, 5, 5, 5, 5, 5, 6, 6, 6, 6, 6, 6, 6, 6, 6, 6, 6, 7, 7, 7, 7, 7, 7, 7, 7, 7, 7, 7, 7, 7, 7, 7, 8, 8, 8, 8, 8, $\cdots$   & - \\ \hline
$\omega(S)$ & 0, 1, 1, 2, 1, 2, 3, 1, 2, 2, 3, 4, 1, 2, 2, 3, 3, 4, 5, 1, 2, 2, 3, 2, 3, 4, 3, 4, 5, 6, 1, 2, 2, 3, 2, 3, 4, 3, 3, 4, 5, 4, 5, 6, 7, 1, 2, 2, 3, 2, $\cdots$   & - \\ \hline
$W_v(S)$ & 1, 1, 1, 2, 1, 2, 3, 1, 2, 3, 4, 6, 1, 2, 3, 4, 5, 7, 10, 1, 2, 3, 4, 4, 6, 8, 7, 10, 14, 20, 1, 2, 3, 4, 4, 6, 8, 7, 8, 11, 15, 13, 18, 25, 35, 1, 2, 3, 4, 4, $\cdots$   & -\\ \hline
$W_e(S)$ & 0, 1, 1, 2, 1, 3, 6, 1, 3, 4, 7, 12, 1, 3, 5, 8, 11, 18, 30, 1, 3, 5, 8, 6, 12, 19, 15, 24, 38, 60, 1, 3, 5, 8, 7, 13, 20, 16, 19, 30, 46, 37, 58, 90, 140, 1, 3, 5, 8, 7, $\cdots$   & -\\ \hline
$\Delta(S)$ & 0, 1, 2, 2, 2, 3, 3, 2, 3, 4, 4, 4, 2, 3, 4, 4, 5, 5, 5, 2, 3, 4, 4, 4, 5, 5, 6, 6, 6, 6, 2, 3, 4, 4, 4, 5, 5, 5, 6, 6, 6, 7, 7, 7, 7, 2, 3, 4, 4, 4, $\cdots$   & -\\ \hline
$|P^H(S)|$ & 1, 1, 1, 2, 1, 3, 6, 1, 4, 6, 12, 24, 1, 5, 10, 20, 30, 60, 120, 1, 6, 15, 30, 20, 60, 120, 90, 180, 360, 720, 1, 7, 21, 42, 35, 105, 210, 140, 210, 420, 840, 630, 1260, 2520, 5040, 1, 8, 28, 56, 56,  $\cdots$   & A078760\\ \hline
$|V_E(S)|$ & 1, 1, 2, 2, 2, 3, 4, 3, 4, 5, 6, 8, 3, 5, 6, 8, 9, 12, 16, 4, 6, 8, 10, 8, 12, 16, 14, 18, 24, 32, 4, 7, 9, 12, 10, 15, 20, 16, 18, 24, 32, 27, 36, 48, 64, 5, 8, 11, 14, 12, $\cdots$   & - \\ \hline
$|V_O(S)|$ & 0, 1, 1, 2, 2, 3, 4, 2, 4, 4, 6, 8, 3, 5, 6, 8, 9, 12, 16, 3, 6, 7, 10, 8, 12, 16, 13, 18, 24, 32, 4, 7, 9, 12, 10, 15, 20, 16, 18, 24, 32, 27, 36, 48, 64, 4, 8, 10, 14, 12, $\cdots$   & -\\ \hline
$|E_E(S)|$ & 0, 1, 1, 2, 2, 4, 6, 2, 5, 6, 10, 16, 3, 7, 9, 14, 17, 26, 40, 3, 8, 11, 18, 12, 23, 36, 27, 42, 64, 96, 4, 10, 14, 22, 16, 30, 46, 32, 38, 58, 88, 68, 102, 152, 224, 4, 11, 16, 26, 19,  $\cdots$   & -\\ \hline
$|E_O(S)|$ & 0, 0, 1, 2, 1, 3, 6, 2, 5, 6, 10, 16, 2, 6, 8, 14, 16, 26, 40, 3, 8, 11, 18, 12, 23, 36, 27, 42, 64, 96, 3, 9, 13, 22, 15, 29, 46, 32, 37, 58, 88, 67, 102, 152, 224, 4, 11, 16, 26, 19,  $\cdots$   & -\\ \hline

$|E^T(S)|$ & 0, 1, 3, 5, 6, 12, 19, 10, 22, 27, 42, 65, 15, 35, 48, 74, 90, 138, 211, 21, 51, 75, 115, 84, 156, 238, 189, 288, 438, 665, 28, 70, 108, 165, 130, 240, 365, 268, 324, 492, 746, 594, 900, 1362, 2059, 36, 92, 147, 224, 186,  $\cdots$   & - \\ \hline
$|P^T(S)|$ & 1, 1, 2, 3, 4, 8, 13, 8, 20, 26, 44, 75, 16, 48, 76, 132, 176, 308, 541, 32, 112, 208, 368, 252, 604, 1076, 818, 1460, 2612, 4683, 64, 256, 544, 976, 768, 1888, 3408, 2316, 3172, 5740, 10404, 7880, $\cdots$   & -  \\ \hline
\end{tabular} }
\end{table}

%\begin{acknowledgements}\label{ackref}
%The \verb"acknowledgements" environment may be used to acknowledge
%indebtedness to colleagues, host institutions and referees. Accounts
%of grants and financial support should be made as a footnote on the
%title page using the \verb"\extraline{}" command in the preamble.
%\end{acknowledgements}

\affiliationone{% in this example, two authors share an institution
  S.-H. Cha\\
  Computer Science Department\\
  Pace University\\
  One Pace Plaza, New York, NY 10038 USA\\
   \email{scha@pace.edu}}
\affiliationtwo{% in this example, one author has two addresses}
   E. G. DuCasse and L.V. Quintas\\
  Mathematics Department\\
  Pace University\\
  One Pace Plaza, New York, NY 10038 USA\\
   \email{educasse@pace.edu\\
     lvquintas@gmail.com}}

\begin{thebibliography}{9}% Replace 9 by 99 if 10 or more references
%
% Please note the use of "\and" between author names below
%

%
\bibitem{AS1972}
 {\bibname M. Abramowitz \and I. A. Stegun},
 {\em Handbook of Mathematical Functions}, National Bureau of Standards, Applied Math. Series 55, Tenth Printing, 1972, p. 831.


\bibitem{Andrews2004}
{\bibname G. E. Andrews \and  K. Eriksson}, 
{\em Integer Partitions} Cambridge Univ. Press, 2004. page18 Exer. 21,22

\bibitem{Comtet1974}
{\bibname L. Comtet}, 
{\em Advanced Combinatorics}, D. Reidel Publishing, 1974, p. 126

%
\bibitem{Harary1972}
{\bibname F. Harary}, 
{\em Graph Theory} 3th ed. Addison-Wesley, Reading, Massachusetts, 1972.


%
\bibitem{HW1979}
 {\bibname G. H. Hardy \and  E. M. Wright},
 {\em An Introduction to the Theory of Numbers} $\S$22.10 and 22.11, 5th ed. Oxford, England: Clarendon Press, 1979.


%
\bibitem{KKW1993}
{\bibname A. Knopfmacher, J. Knopfmacher \and R. Warlimont}, `Ordered factorizations for integers and arithmetical semigroups', {\em Advances in Number Theory. } (Proc. 3rd Conf. of Canadian Number Theory Assoc., 1991), pp. 151-165.Clarendon Press, Oxford, 1993.


%
\bibitem{MacMahon1919}
 {\bibname P. A. MacMahon},
 `Divisors of numbers and their continuations in the theory of partitions',
 {\em  Proc. London Math. Soc.}, 19 (1919), 75-113.

%
\bibitem{oeis}
{\bibname N. J. A. Sloane}, `The On-Line Encyclopedia of Integer Sequences', {\em
OEIS Foundation Inc.}, http://oeis.org.

\end{thebibliography}
\end{document}